\documentclass[11pt]{article}
\usepackage[dvips]{graphicx}
\usepackage{amssymb}
\usepackage{amsmath}
\usepackage{amsthm}
\usepackage{amsfonts}
\usepackage{amsthm,amscd}
\usepackage{amsbsy}
\usepackage{bm}
\usepackage{url} 
\usepackage{setspace}
\usepackage{float}
\usepackage{psfrag}
\usepackage{cite}
\restylefloat{figure}

\newtheorem{theorem}{Theorem}
\newtheorem{lemma}{Lemma}

\newtheorem{assumption}{Assumption}

\newtheorem{remark}{Remark}
\newtheorem{corollary}{Corollary}
\usepackage{setspace}
\usepackage{anysize}
\marginsize{3.0cm}{3.0cm}{2.0cm}{2.0cm}
\def\be{\begin{equation}}
\def\ee{\end{equation}}
\def\ben{\begin{eqnarray}}
\def\een{\end{eqnarray}}

\newcommand{\la}{\langle}
\newcommand{\ra}{\rangle}

\newcommand{\bone}{\mbox{$\bf 1$}}
\newcommand{\supp}{\text{supp}}
\newcommand{\tA}{\mathtt{A}}
\newcommand{\tB}{\mathtt{B}}

\newcommand{\tD}{\mathtt{D}}
\newcommand{\tR}{\mathtt{R}}

\newcommand{\tK}{\mathtt{K}}
\newcommand{\Ker}{\rm Ker\,}
\newcommand{\bR}{\mathbb{R}}

\newcommand{\balpha}{\mbox{\boldmath{$\alpha$}}}
\newcommand{\bbeta}{\mbox{\boldmath{$\beta$}}}
\newcommand{\bsigma}{\mbox{\boldmath{$\sigma$}}}
\newcommand{\sbsigma}{\mbox{\boldmath{\scriptsize $\sigma$}}}
\newcommand{\sbtheta}{\mbox{\boldmath{\scriptsize $\theta$}}}
\newcommand{\btau}{\mbox{\boldmath{$\tau$}}}
\newcommand{\bphi}{\mbox{\boldmath{$\phi$}}}
\newcommand{\bpsi}{\mbox{\boldmath{$\psi$}}}
\newcommand{\btheta}{\mbox{\boldmath{$\theta$}}}

\newcommand{\cQ}{\mathcal{Q}}
\newcommand{\cL}{\mathcal{L}}
\newcommand{\cM}{\mathcal{M}}
\newcommand{\cV}{\mathcal{V}}
\newcommand{\cT}{\mathcal{T}}
\newcommand{\bx}{\mathbf{x}}

\newcommand{\bz}{\mathbf{z}}

\newcommand{\bN}{\mathbb{N}}

\newcommand{\CN}{\mathcal{N}}

\newcommand{\tM}{\mathtt{M}}
\newcommand{\tW}{\mathtt{W}}

\newcommand{\talpha}{\tilde{\alpha}}
\newcommand{\tsigma}{\tilde{\sigma}}

\title{A Stabilized Mixed Finite Element Method for Thin Plate Splines 
Based on Biorthogonal Systems}
\author{Bishnu P.~Lamichhane
\thanks{School of Mathematical \& Physical Sciences,
Mathematics Building - V127,
University of Newcastle,
University Drive,
Callaghan, NSW 2308, Australia, {\tt Bishnu.Lamichhane@newcastle.edu.au}}
  \; and Markus Hegland\thanks{Centre for Mathematics and its Applications, 
Mathematical Sciences Institute,
Australian National University, Canberra, ACT 0200, Australia, 
\tt{Markus.Hegland@anu.edu.au}}}

\begin{document}
\maketitle
\begin{abstract}
The thin plate spline is a popular tool for the interpolation and 
smoothing of scattered data.
In this paper we propose a novel stabilized mixed finite element method 
for the discretization of thin plate splines. The mixed formulation 
is obtained by introducing the gradient of the smoother 
as an additional unknown.  
Working with a pair of bases for 
the gradient of the smoother and the Lagrange multiplier 
which forms a biorthogonal system, 
we can eliminate these two variables (gradient of the smoother 
and Lagrange multiplier) leading to a 
positive definite formulation. A sub-optimal a priori error estimate 
is proved by using the superconvergence property of 
a gradient recovery operator. 
\end{abstract}
\noindent\textit{Key words: Thin plate splines, scattered data smoothing, 
mixed finite element method, saddle point problem, 
biorthogonal system, a priori estimate}\\
\noindent\textit{AMS subject classification: 65D10, 65D15, 65L60, 41A15}

\section{Introduction}
We propose a new finite element approach for the discretization of the 
the thin plate spline \cite{Duc77,Wah90}, which is one of the most 
popular approach in scattered data fitting. 
Scattered data fitting problems occur in many applications such as  
data mining, reconstruction of geometric models, image 
processing, parameter estimation, optic flow, etc., see \cite{BSJ95,Isk04,Wen05}.
 
Let $\Omega \subset \bR^d$ with $d\in \{2,3\}$ 
be a closed and bounded region with polygonal or polyhedral boundary.
In the following, we use  standard notation for the 
norm and semi-norm of Sobolev spaces \cite{BS94}.
Given a set 
 ${\cal G}=\{\bx_i\}_{i=1}^N$ of scattered points in $\Omega$, 
and a function $r$ on ${\cal G}$ with $z_i=r(\bx_i)$ for $i=1,\cdots,N$, 
the thin plate spline is a smooth function 
$u \in H^2(\Omega)$ \cite{Duc77,Wah90} such that 
\begin{equation}\label{lspline}
F(u) \leq F(v)\quad\text{for all}\quad v \in  H^2(\Omega)
\end{equation}
where  
\begin{equation}\label{func}
F(u)=\sum_{i=1}^N (u(\bx_i)-z_i)^2+\alpha \int_{\Omega} 
\sum_{|\nu|=2} {{2}\choose{\nu}} (D^{\nu}u)^2 \, d\bx,
\end{equation}
$\nu=(\nu_1,\cdots,\nu_d)\in \bN_0^d$ is a multi-index,
$|\nu|=\sum_{i=1}^d\nu_i$, and $\alpha$ is a positive constant.

A conventional approach is to use radial basis functions to 
approximate the space $H^2(\Omega)$ in \eqref{lspline}, which 
leads to a dense system matrix. The solution of such a system 
is very expensive when a large data set has to be modelled. 
In this paper we propose an efficient discretization technique 
for the minimization of the functional \eqref{lspline}. 
The basic idea of a finite element method is to minimize 
the functional $F$ given by \eqref{func} over a finite-dimensional function space.
If we want to discretize the minimization problem using a
conforming approach, we need to construct a discrete finite element space 
which is a subset of the Sobolev space $H^2(\Omega)$. 
Construction of such a finite element space is expensive 
\cite{Cia78,BS94}. The class of standard non-conforming finite elements 
\cite{Cia78,Bra01} provides a more efficient discretization 
than the conforming approach. However, their implementation 
requires a more complicated data structure, and 
a suitably constructed mixed formulation provides 
a more efficient and flexible discretization than the non-conforming 
approach. Therefore,  following a similar approach 
as in \cite{JP82,AHR98,CHH00,Ram02}, we modify the original 
minimization problem \eqref{lspline} so that 
the minimization is done over the Sobolev space  $H^1(\Omega)$ 
rather than over the Sobolev space  $H^2(\Omega)$,
and the formulation allows an efficient mixed finite element discretization.
A similar idea has been exploited in \cite{CR74,Fal78,Cia78,Mon87,Lam11b}
for the solution of biharmonic equation with 
simply supported and clamped boundary condition.

The rest of the paper is organized as follows. 
In the remainder of this section, we fix some notation and introduce an 
alternative equivalent variational problem. The next section 
introduces a finite element solution of the problem. We recast 
the problem as a saddle point problem, and 
discuss its algebraic structure. This  motivates us the usage of 
a pair of finite element 
bases (for the gradient of the smoother and the Lagrange multiplier)
which forms a biorthogonal system. Section~\ref{apriori} is devoted 
to the analysis of the discrete problem. Eliminating 
the gradient and the Lagrange multiplier, we get a positive 
definite formulation of the saddle point problem for which we prove 
the existence of a unique solution. The final part of 
Section~\ref{apriori} shows the sub-optimal convergence of our finite element solution to 
the continuous solution. We conclude the paper with a summary. 

Let the Sobolev space $H^1(\Omega)\times [H^1(\Omega)]^d$ be denoted by $\cV$,
and  for two matrix-valued functions 
$\balpha:\Omega \rightarrow \bR^{d\times d}$ and
$\bbeta:\Omega \rightarrow \bR^{d\times d}$,
the Sobolev inner product be defined as 
\[  ( \balpha,\bbeta )_{H^k(\Omega)}:={\sum_{i=1}^d\sum_{j=1}^d (\alpha_{ij},\beta_{ij})_{H^k(\Omega)}},
\] 
where $(\balpha)_{ij}=\alpha_{ij},\; (\bbeta)_{ij}=\beta_{ij}$
with $\alpha_{ij}, \beta_{ij} \in  H^k(\Omega)$, 
and the norm $\|\cdot\|_{H^k(\Omega)}$ is induced from this inner product. 
For $k=0$, an equivalent notation 
\[  ( \balpha,\bbeta )_{L^2(\Omega)}:=\sum_{i=1}^d\sum_{j=1}^d
\int_{\Omega} \alpha_{ij}\beta_{ij}\,dx=\int_{\Omega}\balpha:\bbeta\,dx \]
for the $L^2$-inner product will be used and the $L^2$-norm
 $\|\cdot\|_{L^2(\Omega)}$ is induced by this inner product.

A new formulation of the functional $F$ in \eqref{lspline} is obtained by 
introducing an auxiliary variable 
$\bsigma = \nabla u $ such that the minimization problem   
\eqref{lspline} is rewritten as \cite{JP82,CHH00}
\begin{equation}\label{mspline} 
\min_{\substack{(u,\sbsigma) \in \cV \\ \sbsigma=\nabla u}} G(u,\bsigma)\, ,
\end{equation}
where 
\begin{equation*}
G(u,\bsigma)= \sum_{i=1}^N (u(\bx_i)-z_i)^2+\alpha \|\nabla \bsigma\|_{L^2(\Omega)}^2 .
\end{equation*}

\section{Finite element problem}\label{fesmooth}
Let $\cT_h$ be a quasi-uniform partition of the domain  $\Omega$ in 
$d$-simplices having the mesh-size $h$. 
Let $\hat{T}$ be a reference triangle  defined as 
\[ \hat T:=\{(x,y) :\, 0< x,0< y,x+y< 1\},\] or 
a reference tetrahedron defined as 
\[\hat T:=\{(x,y,z) :\, 0< x,0< y,0<z, x+y+z< 1\}.\]
 
The finite element space is defined by the affine map $F_T$ from the 
reference triangle or tetrahedron 
 $\hat{T}$ to a physical triangle or tetrahedron $T\in \cT_h$. 
Let $ \hat\cL(\hat T)$ and  $ \hat\cQ (\hat T)$ 
be spaces of linear and quadratic polynomials on $\hat T$, 
respectively. Then the finite element space  based on the mesh $\cT_h$ 
is defined as  the space of continuous functions whose restrictions 
to an element $T$ are obtained by an affine map  
from the reference element $\hat T$; that is,
\begin{equation}\label{fespacel}
\cL_h:=\left \{v_h \in H^1(\Omega):\,
v_h|_T=\hat{v}_h\circ F^{-1}_{T},\ \ \hat{v}_h\in
\hat\cL(\hat{T}),\;T\in \cT_h\right \},
\end{equation}
and 
\begin{equation}\label{fespaceq}
\cQ_h:=\left \{v_h \in H^1(\Omega):\,
v_h|_T=\hat{v}_h\circ F^{-1}_{T},\ \ \hat{v}_h\in
\hat\cQ(\hat{T}),\;T\in \cT_h\right \},
\end{equation}
see \cite{Cia78,BS94,Bra01}.

Let  $\cM_h\subset L^2(\Omega)$ be a piecewise 
polynomial space based on  $ \cT_h$ 
 satisfying the following assumptions.
\begin{assumption}\label{A1A2}
\begin{itemize} 
\item[\ref{A1A2}(i)] $\dim \cM_h = \dim \cL_h$.
\item[\ref{A1A2}(ii)] There is a constant $\beta>0$ independent of 
the triangulation $\cT_h$ such that 
\begin{eqnarray}
\|\phi_h\|_{L^2(\Omega)} \leq \beta \sup_{\mu_h \in \cM_h \backslash\{0\}} 
\frac{\int_{\Omega} \mu_h\phi_h\,d\bx} {\|\mu_h\|_{L^2(\Omega)}},
\quad \phi_h \in \cL_h.
\end{eqnarray}
\item[\ref{A1A2}(iii)] The space $\cM_h$ has the approximation property:
\begin{equation}
\inf_{\lambda_h \in \cM_h}\|\phi-\lambda_h\|_{L^2(\Omega)}\leq Ch |\phi|_{H^1(\Omega)},\quad \phi \in H^1(\Omega).
\end{equation}
\end{itemize}
\end{assumption} 
As an example, we can have  $ \cM_h= \cL_h \subset H^1(\Omega)$.
However, we want to utilize the flexibility that 
$\cM_h \subset L^2(\Omega)$ to obtain an efficient finite element  
scheme. 

To obtain the discrete form of the minimization problem 
\eqref{mspline}, we introduce a finite element space $\cV_h$, which is  
a discrete counterpart of $\cV$ as $\cV_h= \cL_h\times [\cL_h]^d$ or 
$\cV_h= \cQ_h \times [\cL_h]^d$. 
Replacing the space $\cV$ in \eqref{mspline} by our discrete space 
$\cV_h$, our discrete problem is to find 
\begin{equation}\label{dspline} 
\min_{(u_h,\sbsigma_h) \in \cV_h}  \sum_{i=1}^N (u_h(\bx_i)-z_i)^2+\alpha 
\|\nabla \bsigma_h\|_{L^2(\Omega)}^2\, 
\end{equation}
subject to \begin{equation}\label{constr}
\la\bsigma_h,\btau_h \ra_{L^2(\Omega)}=\la \nabla u_h, \btau_h \ra_{L^2(\Omega)}, \;\btau_h \in [\cM_h]^d. 
\end{equation}
If we modify the constraint \eqref{constr} to  
\begin{equation*}
\la \nabla u_h, \nabla v_h \ra_{L^2(\Omega)}=
\la\bsigma_h,\nabla v_h \ra_{L^2(\Omega)}, \;v_h \in \cL_h, \end{equation*}
we obtain the finite element thin plate spline presented 
in \cite{AHR98,RHA03}. There are two drawbacks of the finite element 
thin plate spline presented in \cite{AHR98,RHA03}. The first one being 
the saddle point structure of the system matrix arising from the 
discretization which is difficult to solve.
The second drawback is that it does not necessarily 
converge to the standard thin plate spline 
although it has similar smoothing properties as 
the standard thin plate spline \cite{RHA03}. Our goal here is to 
obtain a true approximation of the standard thin plate spline. 

Now we introduce a saddle point formulation of 
the  approach, which can be shown to be  equivalent to 
the minimization problem \eqref{dspline} by using the ideas in \cite{Cia78,BF91}.
We denote the vector of  function values of $u\in C^0(\Omega)$
at the measurement points $\bx_1,\bx_2,\cdots,\bx_N$ by  $Pu\in\bR^N$, i.e.,
\[Pu=(u(\bx_1),u(\bx_2),\cdots,u(\bx_N))^T.\]
Introducing a Lagrange multiplier $\bphi_h$, the variational 
saddle point formulation of the minimization problem \eqref{dspline} 
is to find $((u_h,\bsigma_h),\bphi_h) \in \cV_h \times [\cM_h]^d$ 
so that  
\begin{equation}\label{saddle}
\begin{array}{ccccccccc}
\tilde A((u_h,\bsigma_h),(v_h,\btau_h))&+&B(\bphi_h,(v_h,\btau_h))&=&f(v_h),\; 
&(v_h,\btau_h) &\in& \cV_h, \\
B(\bpsi_h,(u_h,\bsigma_h))&&&=&0,\; &\bpsi_h &\in & [\cM_h]^d,
\end{array}
\end{equation}
 where bilinear forms $\tilde A(\cdot,\cdot)$,
$B(\cdot,\cdot)$ and $f(\cdot)$ are given by
\begin{eqnarray*}
\tilde A((u_h,\bsigma_h),(v_h,\btau_h))&=& (Pu_h)^TPv_h+\alpha \int_{\Omega} \nabla \bsigma_h:\nabla \btau_h\,d\bx,\\
B(\bpsi_h,(v_h,\btau_h))&=&\int_{\Omega} \btau_h\cdot\bpsi_h\,d\bx -\int_{\Omega}
 \nabla v_h\cdot\bpsi_h\, d\bx,\;\text{and}\\  f(v_h)&=&(Pv_h)^T\bz.
\end{eqnarray*}
We recall that the mixed formulation of our problem is closely related to 
the mixed formulation of the Mindlin--Reissner plate \cite{BF91,AB93,Bra96,AF97}, and hence we use some of the ideas presented in \cite{BF91,AB93} to 
analyze our problem.
The existence and uniqueness of the solution of the saddle point problem  
\eqref{saddle} is performed by using the theory presented in \cite{BF91,AB93}. 
The main difficulty here as well as in the context of 
the Mindlin--Reissner plate is that the bilinear form 
$\tilde A(\cdot,\cdot)$ is not coercive on the whole space $\cV_h$. 
However, it would be sufficient that 
the bilinear from $\tilde A(\cdot,\cdot)$ is 
coercive on the space $\Ker B_h$ defined as 
\begin{equation}\label{kernels}
\Ker B_h:= \left\{(v_h,\btau_h)\in \cV_h:\;
\int_{\Omega}(\btau_h-\nabla v_h)\cdot \bpsi_h\,dx=0,\;\bpsi_h \in [\cM_h]^d
\right\}.
\end{equation}
For $\cL_h$ as defined by \eqref{fespacel} and $\cM_h$ satisfying 
Assumptions \ref{A1A2}(i)--\ref{A1A2}(iii), we cannot obtain 
coercivity of $\tilde A(\cdot,\cdot)$ even on the space 
 $\Ker B_h$. This gives us a motivation to modify the 
bilinear form $\tilde A(\cdot,\cdot)$ consistently by adding a 
stabilization term so that 
we obtain the coercivity on the space $\Ker B_h$. 
The modification of the bilinear form $\tilde A(\cdot,\cdot)$ 
is done as suggested by Arnold and Brezzi \cite{AB93} for the Mindlin--Reissner
plate so that our discrete saddle point problem is  
to find $((u_h,\bsigma_h),\bphi_h) \in \cV_h \times [\cM_h]^d$ 
such that  
\begin{equation}\label{dsaddle}
\begin{array}{ccccccccc}
A((u_h,\bsigma_h),(v_h,\btau_h))&+&B(\bphi_h,(v_h,\btau_h))&=&f(v_h),\; 
&(v_h,\btau_h) &\in& \cV_h, \\
B(\bpsi_h,(u_h,\bsigma_h))&&&=&0,\; &\bpsi_h &\in & [\cM_h]^d,
\end{array}
\end{equation}
 where the bilinear form $A(\cdot,\cdot)$ is defined as 
\begin{eqnarray*}
A((u_h,\bsigma_h),(v_h,\btau_h))= (Pu_h)^TPv_h+\alpha \int_{\Omega} \nabla \bsigma_h:\nabla \btau_h\,d\bx+ r\int_{\Omega} (\bsigma_h - \nabla u_h)\cdot (\btau_h-\nabla v_h)\,d\bx
\end{eqnarray*}
with $r>0$ being a parameter.
Since the stabilization term is consistent, the 
parameter $r>0$ can be arbitrary in principle.
By choosing an appropriate parameter, 
the stabilization can, in addition, accelerate the solver 
as in an augmented Lagrangian formulation \cite{BL97}.
Since we do not focus on this aspect of the problem, we simply 
put $r=1$ in the rest of the paper.
After putting $r=1$, we have \[ 
A((u_h,\bsigma_h),(v_h,\btau_h))= \tilde A((u_h,\bsigma_h),(v_h,\btau_h))
+ \int_{\Omega} (\bsigma_h - \nabla u_h)\cdot (\btau_h-\nabla v_h)\,d\bx.
\]

Here our interest is to eliminate the degree of freedom corresponding to 
$\bsigma_h$ and  $\bphi_h$ and arrive at a formulation only depending on $u_h$.
This will dramatically reduce the size of the system matrix, and 
which after elimination of these variables will be 
positive definite. It is well-known that an efficient numerical technique 
can be applied to solve a positive definite system. 

We  now closely look at the algebraic formulation of the problem. 
In the following, we use the same notation  $u_h$, $\bsigma_h$ and $\bphi_h$ 
for the vector representation of the solutions and the solutions as elements in $\cL_h$, $[\cL_h]^d$ and 
$[\cM_h]^d$.
Let $\tR$, $\tA$, $\tB$, $\tW$, $\tK$, $\tD$ and $\tM$ 
be the matrices  associated with the bilinear 
forms $(Pu_h)^TPv_h$,
$\int_{\Omega} \nabla \bsigma_h: \nabla \btau_h\,d\bx$, 
$\int_{\Omega} \nabla u_h\cdot\psi_h\,d\bx$,
$\int_{\Omega} \nabla u_h\cdot\btau_h \, d\bx$,
$\int_{\Omega}\nabla u_h \cdot \nabla v_h\,d\bx$, 
$\int_{\Omega} \bsigma_h \cdot \bpsi_h\,d\bx$
and  $\int_{\Omega} \bsigma_h\cdot\btau_h\,d\bx$, 
respectively. The matrix $\tD$ associated with the bilinear form 
$\int_{\Omega} \bsigma_h \cdot \bpsi_h\,d\bx$ is often called a Gram matrix.
In case of the saddle point formulation, $u_h$, $\bsigma_h$ and $\bphi_h$ are 
three independent unknowns. 
 Letting the test functions $\btau_h$ and $v_h$ to be zero subsequently in 
the first equation of \eqref{dsaddle}, we have 
\begin{equation*}
\begin{array}{cccccc}
(Pu_h)^TPv_h - \int_{\Omega} \nabla v_h\cdot\bphi_h\, d\bx -
\int_{\Omega} (\bsigma_h -\nabla u_h)\cdot \nabla v_h\, d\bx 
&= &f(v_h),\; &v_h&\in &\cL_h,\\
\alpha \int_{\Omega} \nabla \bsigma_h:\nabla \btau_h\,d\bx + 
\int_{\Omega} \bphi_h\cdot\btau_h\,d\bx+\int_{\Omega} (\bsigma_h -\nabla u_h)\cdot \btau_h\, d\bx 
& =& 0,\; &\btau_h  &\in&  [\cL_h]^d.
\end{array}
\end{equation*}
Then the saddle point problem 
\eqref{dsaddle} can be written as the linear system
\begin{equation} \label{saddlealg}
\left[\begin{array}{cccc} 
 \tR+\tK & -\tW^T &-\tB^T \\
-\tW & \alpha \tA+\tM & \tD^T \\
-\tB & \tD & 0
 \end{array} \right]
\left[\begin{array}{ccc} u_h \\ \bsigma_h \\ \bphi_h 
\end{array}\right]=
\left[\begin{array}{ccc} f_h\\0\\0 \end{array}\right],
\end{equation}
where $f_h$ is the vector form of discretization of the linear form 
$f(\cdot)$. Since our goal is to obtain an efficient numerical scheme, 
we want to statically condense out the degree of freedom associated 
with $\bsigma_h$ and $\bphi_h$. This can be achieved easily if 
$\tD$ is invertible and diagonal leading to  a system for $u_h$ only.

Let  $\{\varphi_1,\cdots,\varphi_n\}$ 
be the standard nodal finite element basis of $\cL_h$. 
We define a space $\cM_h$ spanned by the basis 
$\{\mu_1,\cdots,\mu_n\}$, where the basis functions of $\cL_h$ and 
$\cM_h$ satisfy a condition of biorthogonality relation
\begin{eqnarray} \label{biorth}
  \int_{\Omega} \mu_i \ \varphi_j \,d\bx = c_j \delta_{ij},
\; c_j\neq 0,\; 1\le i,j \le n,
\end{eqnarray}
 where $n := \dim \cM_h = \dim \cL_h$,  $\delta_{ij}$ is 
 the Kronecker symbol, and $c_j$ a positive scaling factor.
This scaling factor $c_j$ is chosen to be proportional 
to the area $|\supp \varphi_j|$.
In the following, we give these basis functions for linear 
simplicial finite elements in two and three dimensions. 
For the reference triangle 
$\hat T:=\{(x,y) :\, 0< x,0< y,x+y< 1\}$, we have 
\begin{eqnarray*}
  \hat \mu_1:=3-4x-4y,\,
  \hat\mu_2:=4x-1,\;\text{and}\;
  \hat\mu_3:=4y-1,
\end{eqnarray*}
where the basis functions $\hat \mu_1$, $\hat \mu_2$ and 
 $\hat \mu_3$ are associated with three vertices 
$(0,0)$, $(1,0)$ and $(0,1)$ of the reference triangle. 
For the reference tetrahedron $\hat T:=\{(x,y,z):\,0< x,0<y,0<z,
x+y+z<1\}$, we have 
\begin{eqnarray*}
  \hat \mu_1:=4-5x-5y-5z,\,
  \hat\mu_2:=5x-1,\;\text{and}\;
  \hat\mu_3:=5y-1,\,
\hat\mu_4:=5z-1,
\end{eqnarray*}
where the basis functions $\hat \mu_1$, $\hat \mu_2$, $\hat \mu_3$ and 
 $\hat \mu_4$ associated with four vertices 
$(0,0,0)$, $(1,0,0)$, $(0,1,0)$ and $(0,0,1)$ of the reference tetrahedron. 
The global basis functions for the test space are constructed by 
glueing the local basis functions together and thus the 
assembling process is exactly the same as in the standard 
finite element method. 
 
These global basis functions then satisfy the condition of 
biorthogonality \eqref{biorth} with global finite element basis functions.
 As these functions in $\cM_h$ are 
defined exactly in the same way as the finite element basis functions in 
$\cL_h$, they satisfy $\supp \mu_i=\supp \varphi_i$ for 
$i=1,\cdots, n$. After statically condensing out variables 
$\bsigma_h$ and $\bphi_h$ (block elimination), we arrive at a reduced system 
\[\left((\tR+\tK)-(\tW^T\tD^{-1}\tB+\tB^T\tD^{-1}\tW)+
\tB^T\tD^{-1}(\alpha \tA+\tM)\tD^{-1}\tB\right)u_h = f_h.
\]
\begin{remark}\label{rem1}
Such biorthogonal basis functions are very popular in 
the context of mortar finite elements \cite{BWHabil,KLP01,Lam06}.
Construction of local basis functions 
of the space $\cM_h$ satisfying all three 
Assumptions \ref{A1A2}(i)--\ref{A1A2}(iii) as well as 
the biorthogonality condition \eqref{biorth} 
for different finite element spaces can be found in 
\cite{BWHabil,LSW05,LW07}. 
Working with nodal finite element basis functions based on 
Gauss--Lobatto quadrature nodes for rectangular or hexahedral 
triangulation, we have shown the construction of local basis functions 
of $\cM_h$ satisfying all these assumptions for an arbitrary 
order finite element space \cite{LW07}. 
\end{remark}
\section{An a priori error estimate}\label{apriori}
In the previous section, we have shown how the degree of freedom for the 
gradient and Lagrange multipliers can be eliminated from the 
linear system \eqref{saddlealg}. Now we 
want to eliminate the gradient of the smoother $\bsigma_h$ and Lagrange 
multiplier $\bphi_h$ from the saddle point problem \eqref{dsaddle}. 
To this end, we introduce a quasi-projection 
operator: $R_h:L^2(\Omega)\rightarrow \cL_h$, which is defined as 
\[ \int_{\Omega} R_hv\, \mu_h \,d\bx = \int_{\Omega} v \mu_h \,d\bx, 
\; v \in L^2(\Omega),\; \mu_h \in \cM_h. \]
This type of operator is introduced in \cite{SZ90} to obtain  
the finite element interpolation of non-smooth functions satisfying 
boundary conditions, and
is used in \cite{BMP94} in the context of mortar finite elements. 
The definition of $R_h$ allows us to write the weak gradient as
\[ \bsigma_h= R_h\nabla u_h, \]
where the operator $R_h$ is applied to the vector $\nabla u_h$ componentwise.
We see that $R_h$ is well-defined due to Assumptions 
\ref{A1A2}(ii). Furthermore, the restriction of 
$R_h$ to $\cL_h$ is the identity. Hence $R_h$ is a projection onto the 
space $\cL_h$. We note that $R_h$ is not the orthogonal projection onto 
$\cL_h$ but an oblique projection onto $\cL_h$. 
Oblique projectors are studied extensively in \cite{Gal03}, and 
different expressions for the norm of oblique projections are provided 
in \cite{Szy06}. 
According to the biorthogonality relation between the basis functions of 
$\cL_h$ and $\cM_h$ \eqref{biorth},
the action of operator $R_h$ on a function $v \in L^2(\Omega)$ 
can be written as 
\begin{equation}\label{eq2}
 R_hv = \sum_{i=1}^n\frac{\int_{\Omega} \mu_i \,v\,dx}{c_i} \varphi_i,
\end{equation}
and consequently the operator $R_h$ is local in the sense to be 
given below, see also \cite{AO00}.
Let $S(T')$ be the patch of an element $T' \in {\cal T}_h$ 
which is the interior of the closed set 
\begin{equation} \label{patch}
\bar S(T')=\overline{\bigcup{\{{ T} \in {\cal T}_h : 
 \partial T \cap \partial T' \neq \emptyset\}}}.
\end{equation}
Then $R_h$ is local in the sense that for any $v \in L^2(\Omega)$, 
the value of $R_hv$ at any point in $T\in \cT_h$ only 
depends on the values of $v$ in $S(T)$ \cite{AO00}.
In the following, we will use a generic constant 
$C$, which will take different values at different places but will be 
always independent of the mesh-size $h$.
The stability of $R_h$ in $L^2$-norm is shown in the following lemma 
\cite{KLP01}.
\begin{lemma} \label{lemma1}
Under Assumption \ref{A1A2}(ii), there exists $C>0$ such that 
 \begin{equation}\label{l2s}
\|R_h v\|_{L^2(\Omega)}\leq C \|v\|_{L^2(\Omega)}\quad\text{for all}\quad
v \in  L^2(\Omega).
\end{equation}
\end{lemma} 
\begin{proof}
By Assumption \ref{A1A2}(ii)
\begin{eqnarray}
\|R_h v\|_{L^2(\Omega)}\leq \beta \sup_{\mu_h \in \cM_h \backslash\{0\}} 
\frac{\int_{\Omega} \mu_h R_hv\,d\bx} {\|\mu_h\|_{L^2(\Omega)}} = 
\beta \sup_{\mu_h \in \cM_h \backslash\{0\}} 
\frac{\int_{\Omega} \mu_h v\,d\bx} {\|\mu_h\|_{L^2(\Omega)}} \leq 
\beta \|v\|_{L^2(\Omega)}.
\end{eqnarray} 
\end{proof}
In the following, $P_h: L^2(\Omega) \rightarrow \cL_h$ 
will denote the $L^2$-orthogonal projection onto $\cL_h$. It is well-known 
that the operator $P_h$ is stable in both $L^2$- and $H^1$-norms. 
Using the stability of the operator $R_h$ in the $L^2$-norm, and 
of the operator $P_h$ in the $H^1$-norm, we can show that 
$R_h$ is also stable in the $H^1$-norm, see \cite{Lam06} for the 
locally quasi-uniform case.
\begin{lemma} \label{lemma2}
Under Assumption \ref{A1A2}(ii), there exists $C>0$ such that 
\begin{eqnarray*}
|R_h w|_{H^1(\Omega)} 
\leq C|w|_{H^1(\Omega)}\quad\text{for all}\quad w \in  H^1(\Omega).
\end{eqnarray*}
\end{lemma}
\begin{proof}
Using the $L^2$-stability from Lemma \ref{lemma1} and the inverse inequality,
we get for $w\in H^1(\Omega)$
\begin{eqnarray*}
|R_h w|_{H^1(\Omega)} 
&\leq& |R_h w-P_hw|_{H^1(\Omega)}+|P_hw|_{H^1(\Omega)}\\
&\leq& C \left(\frac{1}{h} \|R_h (w-P_hw)\|_{L^2(\Omega)}+|w|_{H^1(\Omega)}\right)\\
&\leq& C \left(\frac{1}{h} \|w-P_hw\|_{L^2(\Omega)}+
|w|_{H^1(\Omega)}\right)
\leq C|w|_{H^1(\Omega)}.
\end{eqnarray*}
\end{proof}
The following lemma establishes the approximation property of operator 
$R_h$ for a function $v \in H^{s}(\Omega)$, see also \cite{Lam06}.
\begin{lemma}\label{lemma3}
Under Assumption \ref{A1A2}(ii), 
there exists a constant $C$ independent 
of the mesh-size $h$ so that for $v \in H^{s+1}(\Omega)$, $0<s\leq 1$, we have
\begin{equation} \label{estn}
\begin{array}{ccc}
\|v-R_hv\|_{L^2(\Omega)} &\leq & C h^{1+s}|v|_{H^{s+1}(\Omega)}\\
\|v-R_hv\|_{H^1(\Omega)} &\leq&  C h^s|v|_{H^{s+1}(\Omega)}.
\end{array}
\end{equation}
\end{lemma}
\begin{proof} 
We start with a triangle inequality
\[ \|v-R_hv\|_{L^2(\Omega)} \leq  \|v-P_hv\|_{L^2(\Omega)}+ 
\|P_hv-R_hv\|_{L^2(\Omega)}. \] 
Since $R_h$ acts as an identity on $\cL_h$, we have
\[ \|v-R_hv\|_{L^2(\Omega)} \leq
\|v-P_hv\|_{L^2(\Omega)}+ \|R_h(P_hv-v)\|_{L^2(\Omega)}.\]
Now we use the $L^2$-stability of $R_h$ from Lemma \ref{lemma1}
 to obtain 
\[\|v-R_hv\|_{L^2(\Omega)} \leq C\|v-P_hv\|_{L^2(\Omega)}.
\]
The first inequality of \eqref{estn} follows by using the approximation 
property of the orthogonal projection $P_h$ onto $\cL_h$, see \cite{Bra01}.
The second inequality of \eqref{estn} is proved similarly using 
the stability of $R_h$ in $H^1$-norm and the approximation 
property of the orthogonal projection $P_h$ onto $\cL_h$.
\end{proof}
Using the property of operator $R_h$, we can eliminate the 
degrees of freedom corresponding to $\bsigma_h$ so that 
the solution $u_h$ of \eqref{dsaddle} satisfies
\begin{equation}\label{dsplinem}
J_{\alpha}(u_h) = \min_{v_h \in \cL_h} J_{\alpha}(v_h),
\end{equation}
where \[
 J_{\alpha}(v_h)=\|Pv_h\|^2+\alpha \|\nabla (R_h \nabla v_h)\|_{L^2(\Omega)}^2
+\|R_h\nabla v_h -\nabla v_h\|^2_{L^2(\Omega)}-2\,(Pv_h)^T\bz.
\] 
In order to show that this problem has a unique solution, 
we define a P-inner product $\la\cdot,\cdot\ra_P$ with 
\[\la u_h,v_h\ra_P=
(Pu_h)^TPv_h+\alpha\int_{\Omega} \nabla\bsigma_h:\nabla\btau_h\,d\bx+
\int_{\Omega}(\bsigma_h-\nabla u_h)\cdot (\btau_h-\nabla v_h)\,d\bx,\] 
where $\bsigma_h = R_h\nabla u_h$ and $\btau_h = R_h\nabla v_h$. 
It follows that 
\[  J_{\alpha}(v_h) = \la v_h,v_h\ra_P - 2\,(Pv_h)^T\bz.\]

The following theorem shows that the P-inner product defines 
an inner product on the vector space $\cL_h$ or $\cQ_h$ 
given by \eqref{fespacel}.
\begin{theorem}\label{th0}
Let $\alpha >0$ and ${\cal G}\subset \bar\Omega$ have at least 
three non-collinear points for $d=2$ and 
 and four non-coplanar points for $d=3$.  
Then the P-inner product defined above is an inner 
product on the vector space $\cL_h$ or $\cQ_h$.
\end{theorem}
\begin{proof}
In order to show that the P-inner product is indeed an inner product, 
we have to prove the following properties of P-inner product:
\begin{itemize}
\item[(1)] $\la v_h,v_h\ra_P \geq 0,\;\text{and}\;  \la v_h,v_h\ra_P= 0\;\text{if and 
only if}\; v_h=0,\quad v_h \in \cL_h$,  
\item[(2)] $\la v_h+w_h,z_h\ra_P =\la v_h,z_h\ra_P+\la w_h,z_h\ra_P,\quad v_h,w_h,z_h\in \cL_h$,
\item[(3)] $\la v_h,bz\ra_P=b\la v_h,z_h\ra_P,\quad v_h\in \cL_h,\;b\in\bR$,
\item[(4)] $\la v_h,w_h\ra_P=\la w_h,v_h\ra_P,\quad v_h,w_h\in \cL_h$.
\end{itemize}
It is trivial to show that the P-inner product 
satisfies the second, third and fourth properties.  
It is also obvious that $\la v_h,v_h\ra_P \geq 0$, and 
$\la v_h,v_h\ra_P=0$ if $v_h=0$. 
It remains to show that  $\la v_h,v_h\ra_P =0$ implies $v_h=0$.
We have $\la v_h,v_h\ra_P=
\|Pv_h\|^2+\alpha\|\nabla\btau_h\|^2_{L^2(\Omega)}+
\|\btau_h-\nabla v_h\|_{L^2(\Omega)}$ with 
$\btau_h=R_h \nabla v_h$. Let $\la v_h,v_h\ra_P=0$. Then, 
$\|Pv_h\|^2=0$, $\|\nabla\btau_h\|^2_{L^2(\Omega)}=0$ and 
$\|\btau_h-\nabla v_h\|_{L^2(\Omega)}=0$
separately as they are all positive. Since $\btau_h$ is continuous, 
$\|\nabla \btau_h\|_{L^2(\Omega)}=0$ if and only if $\btau_h$ is a constant 
vector function in $\Omega$. Similarly, 
$\|\btau_h-\nabla v_h\|_{L^2(\Omega)}=0$ implies that  
$\nabla v_h$ is also constant in $\Omega$, 
and thus $v_h$ is a global linear function in $\Omega$.  
On the other hand, $\|Pv_h\|=0$ implies that $v_h$ is zero on 
${\cal G} \subset \bar\Omega$, which contains at least 
three non-collinear points for $d=2$
or four non-coplanar points for $d=3$.
Hence $v_h$ is a global linear function which is 
zero at three non-collinear points for $d=2$
or four non-coplanar points for $d=3$,
and therefore, identically vanishes in $\Omega$.
\end{proof}
The P-norm of an element $u_h \in \cL_h$ or $u_h \in \cQ_h$ 
induced by the inner product 
$\la\cdot,\cdot\ra_P$ is given by 
$\|u_h\|_P^2=
\|Pu_h\|^2+\alpha \|\nabla R_h\nabla u_h\|_{L^2(\Omega)}^2+
\|R_h\nabla u_h -\nabla u_h\|^2_{L^2(\Omega)}$. 
Let the bilinear form  $a(\cdot,\cdot)$ be defined as 
\begin{eqnarray*}a(u_h,v_h)&=&
(Pu_h)^TPv_h+\alpha \int_{\Omega} \nabla \bsigma_h:\nabla \btau_h\,d\bx+\int_{\Omega}(\bsigma_h-\nabla u_h)\cdot (\btau_h-\nabla v_h)\,d\bx
\end{eqnarray*}
with $\bsigma_h = R_h\nabla u_h$ and $\btau_h = R_h\nabla v_h$.
Since the bilinear form $a(\cdot,\cdot)$ is symmetric, 
the minimization problem \eqref{dsplinem} is equivalent to the 
variational problem of finding $u_h \in \cL_h$ or $u_h \in \cQ_h$ 
 such that \cite{Cia78,Bra01}
\begin{equation}\label{mvarp}
a(u_h,v_h)=f(v_h),\quad v_h \in \cL_h \; \text{or}\; u_h \in \cQ_h.
\end{equation}
Furthermore, the following corollary holds.
\begin{corollary}
Under the assumptions of Theorem \ref{th0}, the variational problem 
\eqref{mvarp} admits a unique solution which 
depends continuously on the data.
\end{corollary}
\begin{proof}
Let $u_h,v_h \in \cL_h$, or $u_h,v_h \in \cQ_h$.
It then follows that 
$|a(u_h,v_h)|\leq \|u_h\|_P\|v_h\|_P$ and $|f(v_h)|\leq C\|v_h\|_P$. 
Moreover, using the definition of P-norm 
$a(v_h,v_h)=\|v_h\|_P$, and thus $a(\cdot,\cdot)$ is coercive
with respect to the norm $\|\cdot\|_P$. 
Hence our variational problem \eqref{mvarp} has a unique 
solution by Lax-Milgram Lemma \cite{Cia78,BS92}.
From the definition of the $P$-inner product, we have 
\[ a(v_h,v_h) = \|v_h\|^2_P\]   
and thus, for the solution $u_h$,  $\|u_h\|^2_P=f(u_h)$. 
\end{proof}
\begin{remark}
Using the unique solution $u_h$ of the variational 
problem \eqref{mvarp}, we have a unique solution $(u_h,\bsigma_h)$ of 
\eqref{dsaddle} with 
$\bsigma_h = R_h \nabla u_h$. 
\end{remark}
The error estimate is obtained in the \emph{energy norm} $\|\cdot\|_A$ 
induced by the bilinear form $A(\cdot,\cdot)$ defined as 
\begin{equation} 
\|(u,\bsigma)\|_A :=\sqrt{\|Pu\|^2+\alpha |\bsigma|^2_{H^1(\Omega)} + 
\|\bsigma-\nabla u\|^2_{L^2(\Omega)}},
\quad (u,\bsigma) \in \tilde V 
\times [H^1(\Omega)]^d,
\end{equation}
where $ \tilde V = C^0(\Omega) \cap H^1(\Omega)$.
The following theorem  is the starting point for the 
a priori error estimate, see also \cite{Cia78,Lam11b}.
\begin{theorem}\label{th1}
Let $u$ be the solution of  problem \eqref{lspline} 
satisfying $u \in H^4(\Omega)$, $\bsigma =\nabla u$ and 
$\bphi=\alpha\Delta\bsigma$,
 and $u_h$ be the solution of 
problem \eqref{mvarp}, and $\bsigma_h = R_h \nabla u_h$. 
Then there exists a constant $C>0$ independent of 
the mesh-size $h$ so that 
\begin{eqnarray*}
\|(u-u_h,\bsigma-\bsigma_h)\|_A \leq 
C\left(\inf_{(w_h,\sbtheta_h)\in \Ker B_h} \|(u-w_h,\bsigma-\btheta_h)\|_A+h |\bphi|_{H^1(\Omega)}\right).
\end{eqnarray*}
\end{theorem}
\begin{proof}
Here $u$, $\bsigma$ and $\bphi$ satisfy \cite{BF91}
 \begin{equation*}
\begin{array}{ccccccccc}
A((u,\bsigma),(v,\btau))&+&B(\bphi,(v,\btau))&=&f(v),\; &(v,\btau) &\in& \cV, \\
B(\bpsi,(u,\bsigma))&&&=&0,\; &\bpsi &\in & [L^2(\Omega)]^d.
\end{array}
\end{equation*}
Let $(w_h,\btheta_h) \in \Ker B_h$ so that 
$(u_h-w_h,\bsigma_h-\btheta_h) \in \Ker B_h$, and  hence 
\begin{eqnarray}\label{aeq1}
 \|(u_h-w_h,\bsigma_h-\btheta_h)\|_A 
\leq \sup_{(v_h,\btau_h)\in \Ker B_h} 
\frac{A((u_h-w_h,\bsigma_h-\btheta_h),(v_h,\btau_h))}
{\|(v_h,\btau_h)\|_A }
\end{eqnarray}
Since $A((u-u_h,\bsigma-\bsigma_h),(v_h,\btau_h))+ B(\bphi,(v_h,\btau_h))=0$ for all $(v_h,\btau_h) \in \Ker B_h$, we have 
\begin{eqnarray}\label{aeq2}
& &A((u_h-w_h,\bsigma_h-\btheta_h),(v_h,\btau_h)\notag\\&&= 
A((u-w_h,\bsigma-\btheta_h),(v_h,\btau_h))+A((u_h-u,\bsigma_h-\bsigma),(v_h,\btau_h))\notag\\
&&=A((u-w_h,\bsigma-\btheta_h),(v_h,\btau_h))+B(\bphi,(v_h,\btau_h)).
\end{eqnarray}
The continuity of $A(\cdot,\cdot)$ yields 
\begin{equation}\label{aeq3}
\left|A((u-w_h,\bsigma-\btheta_h),(v_h,\btau_h)\right|
 \leq \|(u-w_h,\bsigma-\btheta_h)\|_A \|(v_h,\btau_h)\|_A.
\end{equation}
Denoting the orthogonal projection of $\bphi$ onto $[\cM_h]^d$ with 
respect to $L^2$-inner product by $\tilde \bphi_h$, we have 
\begin{equation}\label{estorth}
 B(\bphi,(v_h,\btau_h))=\int_{\Omega} (\btau_h-\nabla v_h) \cdot (\bphi-\tilde \bphi_h)\,dx
\leq C h \|\btau_h-\nabla v_h\|_{L^2(\Omega)} |\bphi|_{H^1(\Omega)}.
\end{equation}
The result then follows by combining 
\eqref{aeq1}, \eqref{aeq2}, \eqref{aeq3} and \eqref{estorth}.
\end{proof}
Two different finite element methods for the 
discrete problem \eqref{dsaddle} are obtained by setting $\cV_h = 
\cQ_h \times [\cL_h]^d$ and  $\cV_h = \cL_h \times [\cL_h]^d$.
We prove suboptimal convergence rate in the energy norm 
$\|\cdot \|_A$ for both cases. 
In the first step, we consider $\cV_h=\cQ_h \times [\cL_h]^d$. 
\begin{theorem} \label{thn1}
Let $\cV_h=\cQ_h \times [\cL_h]^d$. 
Then under the assumptions of Theorem \ref{th1},
there exists $(v_h,\btau_h) \in \Ker B_h$ such that 
\begin{equation}\label{thenq}
\|(u-v_h,\bsigma-\btau_h)\|_A
\leq C h \|u\|_{H^3(\Omega)}.
\end{equation}
\end{theorem}
\begin{proof}
Let $v_h$ be the quadratic Lagrange interpolation of $u$ 
with respect to the mesh $\cT_h$.
Then it is well-known that 
\begin{equation} \label{apeq}
\|u-v_h\|_{H^k(\Omega)} \leq  C h^{2-k} |u|_{H^2(\Omega)},\quad k=0,1.
\end{equation}
Moreover, 
\begin{equation}
\|P(u-v_h)\|^2 \leq C h^2 |u|^2_{H^2(\Omega)}.
\end{equation}
Let us recall the definition of the error in the energy norm 
\[ \|(u-v_h,\bsigma-\btau_h)\|_A =
\sqrt{\|P(u-v_h)\|^2+\alpha |\bsigma-\btau_h|^2_{H^1(\Omega)} + 
\|\bsigma-\btau_h-\nabla u+\nabla v_h\|^2_{L^2(\Omega)}}.
\]
It is now sufficient to show that 
\[ \|\bsigma - \btau_h\|_{H^1(\Omega)} \leq h \|u\|_{H^3(\Omega)}.\] 
Since $u \in H^{3}(S(T))\cap  H^1(\Omega)$, $T \in \cT_h$, we have
\[ 
\| \nabla u-R_h \nabla v_h\|_{L^2(T)} \leq C h^{2}\|u\|_{H^{3}(S(T))},\]
as in \cite{BK03}. 
Hence we have 
\begin{equation}\label{approx}
 \|\bsigma -\btau_h\|_{L^2(\Omega)} \leq 
C h^2  \|u\|_{H^{3}(\Omega)}.
\end{equation}
Now  using a triangle 
inequality, an inverse estimate and projection property of $R_h$, 
we obtain 
\begin{align*}
 \|\bsigma -\btau_h\|_{H^1(\Omega)}  & \leq 
\|\bsigma -R_h\bsigma\|_{H^1(\Omega)}+  \|R_h \bsigma -R_h\nabla v_h\|_{H^1(\Omega)}\\
& \leq C \left(\|\bsigma -R_h\bsigma\|_{H^1(\Omega)}+  \frac{1}{h}
\|R_h \bsigma -R_h\nabla v_h\|_{L^2(\Omega)} \right)\\
& \leq C \left(\|\bsigma -R_h\bsigma\|_{H^1(\Omega)}+  \frac{1}{h}
\| \bsigma -R_h\nabla v_h\|_{L^2(\Omega)} \right).
\end{align*}
The first term in the right has 
 the correct approximation from Lemma \ref{lemma3}, and 
the second term from \eqref{approx}
\end{proof}
Using the results of Theorems \ref{th1} and \ref{thn1}, we 
get the following approximation result for the discrete solution.  
\begin{corollary}
Let $u$ be the solution of continuous problem \eqref{lspline} 
with $u \in H^4(\Omega)$, $\bsigma =\nabla u$ and 
$\bphi=\alpha\Delta\bsigma$, and $u_h$ be that of discrete 
problem \eqref{mvarp} with $\bsigma_h = R_h\nabla u_h$ and 
$\cV_h = \cQ_h \times [\cL_h]^d$
Then there exists a constant $C>0$ independent of 
the mesh-size $h$ so that 
\begin{eqnarray*}
\|(u-u_h,\bsigma-\bsigma_h)\|_A \leq 
Ch\left(\|u\|_{H^3(\Omega)}+|\bphi|_{H^1(\Omega)}\right).
\end{eqnarray*}
\end{corollary}
In order to show the approximation property with 
$\cV_h = \cL_h \times [\cL_h]^d$, we use the 
super-approximation of a gradient recovery operator recently proposed in 
\cite{XZ04}. This idea is utilised in \cite{Lam11b} 
to get a finite element approximation for the biharmonic problem. 
Since the super-approximation property is only available for 
the two-dimensional case, in the following, we assume that $\Omega \subset \bR^2$. 

First we need an assumption on our mesh similar to 
{\em Condition ($\talpha, \tsigma$)} in \cite{XZ04}. 
Let $\CN_h=\{\bx_i\}_{i=1}^{n_v}$ be the set of vertex nodes in $\cT_h$, 
and $S_{i}$ be the support 
of the finite element basis function $\phi_i$ at $\bx_i \in \CN_h$. 
We impose the following assumption on our mesh.
\begin{assumption}\label{ass:mesh}
 \begin{itemize}
\item[(1)] Let $\cT_h=\cT_h^2\cup \cT_h^1$ and 
$\bar \Omega= \bar \Omega^1_h\cup \bar\Omega^2_h$,
such that 
\[ |\Omega^2_h|=O(h^{\tsigma}),\; \tsigma>0,\quad\text{and}\quad 
\bar \Omega^i_h =\cup_{T \in \cT_h^i} \bar T,\; i=1,2.\]
\item[(2)] Choosing $\bx_i$ as the origin of local coordinates, 
  \[\sum\limits_{T \subset S_i} \frac{|T|}{|S_i|}\bz_T = O(h^{1+\talpha})\bone,\; 
\bx_i \in \CN_h \cap \Omega^1_h,\] 
\end{itemize}
where $\bz_T$ is the coordinate vector of the barycenter of element 
$T$, $\talpha>0$, and  
$\bone$ is the $d$-dimensional vector having each component $1$. 
\end{assumption}

If a  mesh is uniformly regular, the  assumption holds with $\talpha = \infty$ 
and $\tsigma = 1$.
That means we are allowing $O(h^{1+\alpha})$ deviation from 
uniformly regular meshes. 
In fact, if two adjacent triangles in $\cT_h$ form 
an $O(h^{1+\alpha})$ approximate parallelogram, this assumption 
is satisfied \cite{XZ04}, where two triangles are 
said to form an $O(h^{1+\alpha})$ approximate parallelogram, 
if the lengths of two opposite edges differ only by 
 $O(h^{1+\alpha})$ \cite{BX03}.

Let $(\nabla I_hu)_{|T}$ be the restriction of $\nabla I_hu$ to an 
element $T \in \cT_h$. 
Then using \eqref{eq2}, we have 
\[ 
R_h (\nabla I_h u)(\bx_i)= \sum_{T \subset S_i} \frac{|T|}{|S_i|}
(\nabla I_hu)_{|T}.\]
The following theorem can be proved exactly as in \cite{XZ04}.
\begin{theorem}  \label{thnew1}
Under Assumption \ref{ass:mesh}, if $ u \in W^{3,\infty}(S_i)$, 
for any $\bx_i \in \CN_h$
\[| (R_h \nabla I_h u)(\bx_i)-(\nabla u)(\bx_i)|\leq C 
\left(h^2+h^{1+\talpha} \right)\|u\|_{W^{3,\infty}(S_i)}. \]
\end{theorem}

Our goal is to prove a super-approximation property of the gradient 
recovery operator $R_h$ as in \cite{XZ04}. 
\begin{theorem} \label{thnew2}
Let $ u \in W^{3,\infty}(\Omega)$, and $I_h u$ be the linear Lagrange 
interpolation of $u$ with respect to $\cT_h$. 
Assume that the triangulation satisfies Assumption \ref{ass:mesh}.
Then \[ 
\|\nabla u - R_h \nabla I_h u\|_{L^2(\Omega)} \leq 
C h^{1+\rho} \|u\|_{W^{3,\infty}}(\Omega),\]
where $\rho = \min(\talpha,\frac{\tsigma}{2})$, 
and $\talpha$ and $\tsigma$ are as in 
Assumption \ref{ass:mesh}.
\end{theorem}
\begin{proof}
Since the Lagrange interpolation operator 
 reproduces all piecewise linear 
polynomials with respect to the mesh $\cT_h$,
\[ \| u- I_hu\|_{L^2(\Omega)} \leq C h^2 |u|_{H^2(\Omega)}.\]
Now we decompose 
\begin{equation*}
\nabla u -R_h \nabla I_h u = 
\nabla u -  I_h \nabla u + I_h \nabla u -   R_h \nabla I_h u 
\end{equation*}
so that a triangle inequality yields 
\begin{equation}\label{trieq}
\|\nabla u -R_h \nabla I_h u\|_{L^2(\Omega)} \leq 
\|\nabla u -  I_h \nabla u \|_{L^2(\Omega)}+ 
\|I_h \nabla u -   R_h \nabla I_h u\|_{L^2(\Omega)}. 
\end{equation}
The approximation property of $ I_h$ yields 
\[ \| \nabla u- I_h\nabla u\|_{L^2(\Omega)} \leq C h^2 |u|_{H^{3}(\Omega)}.\]
Under Assumption \ref{ass:mesh}, we have Theorem \ref{thnew1}, and hence 
 \begin{eqnarray*}
&&\|R_h \nabla I_h u-I_h \nabla u \|_{L^2(\Omega^1_h)}\\&\leq&
C \left(\sum_{T \subset \Omega^1_h} |T| 
\sum_{\bz \in \CN_h \cap \bar T}
|(R_h \nabla I_h u)(\bz)-(\nabla u)(\bz)|^2\right)^{1/2}\\
&\leq& C h^{1+\talpha}  \|u\|_{W^{3,\infty}(\Omega)}\sqrt{|\Omega^1_h|}
\leq C h^{1+\talpha}  \|u\|_{W^{3,\infty}(\Omega)}.
\end{eqnarray*}
Moreover, using Assumption \ref{ass:mesh} again, we get   
 \begin{eqnarray*}
\|R_h \nabla I_h u- I_h \nabla u \|_{L^2(\Omega^2_h)}
\leq C \left(h^2|u|_{W^{3,\infty}(\Omega^2_h)}+ h \|u\|_{W^{3,\infty}(\Omega)}\sqrt{|\Omega^2_h|} \right)\leq C  
h^{1+\frac{\tsigma}{2}} \|u\|_{W^{3,\infty}(\Omega)}.
\end{eqnarray*}
The final result follows from using the above estimates in \eqref{trieq}.
\end{proof}

The following theorem guarantees a sub-optimal convergence rate of 
the finite element approximation under Assumptions \ref{ass:mesh}.

\begin{theorem} \label{th2}
Let $\cV_h = \cL_h \times [\cL_h]^d$. 
Then under the assumptions of Theorem \ref{th1},
there exists $(v_h,\btau_h) \in \Ker B_h$ such that 
\begin{equation}\label{theq}
\|(u-v_h,\bsigma-\btau_h)\|_A
\leq C h^{\rho} \|u\|_{H^3(\Omega)}.,
\end{equation}
where $\rho = \min(\talpha,\frac{\tsigma}{2})$.
\end{theorem}
\begin{proof}
Although the proof of this theorem is similar to that of Theorem \ref{thn1}, 
we give a proof for completeness. 
Let $v_h$ be the Lagrange interpolation of $u$ with respect to the mesh 
$\cT_h$ using linear finite elements.
Then it is well-known that 
\begin{equation}
\|u-v_h\|_{H^k(\Omega)} \leq C h^{2-k} |u|_{H^2(\Omega)},\quad k=0,1.
\end{equation}
Moreover, by Sobolev embedding 
\begin{equation}
\|P(u-v_h)\|^2 \leq C h^2 |u|^2_{H^2(\Omega)}.
\end{equation}
Let us recall the definition of the error in the energy norm 
\[ \|(u-v_h,\bsigma-\btau_h)\|_A =
\sqrt{\|P(u-v_h)\|^2+\alpha |\bsigma-\btau_h|^2_{H^1(\Omega)} + 
\|\bsigma-\btau_h-\nabla u+\nabla v_h\|^2_{L^2(\Omega)}}.
\]
Let $\btau_h = R_h \nabla v_h$ so that $(v_h,\btau_h) \in \Ker B_h$.
The approximation property of operator $R_h$ given by 
Theorem \ref{thnew1} yields
\begin{equation}
\|\nabla u-R_h \nabla v_h\|_{L^2(\Omega)} \leq C h^{1+\rho}\|u\|_{H^{3}(\Omega)}.
\end{equation}
Hence, it suffices to show that 
\[ \|\bsigma -\btau_h\|_{H^1(\Omega)} \leq  C h^{\rho} \|u\|_{H^3(\Omega)}.
\]
Since $\bsigma =\nabla u$ and $\btau_h = R_h \nabla v_h$,
\begin{equation}\label{triest}
 \|\bsigma -\btau_h\|_{H^1(\Omega)}  \leq 
\|\bsigma -R_h\bsigma\|_{H^1(\Omega)}+  \|R_h \bsigma -R_h\nabla v_h\|_{H^1(\Omega)}.
\end{equation}
The first term in the right-hand side of \eqref{triest} has 
the correct approximation from Lemma \ref{lemma3}.
To estimate the second term, we use $\bsigma = \nabla u$ and 
apply an inverse estimate to get 
\begin{eqnarray*}
  \|R_h \bsigma -R_h\nabla v_h\|_{H^1(\Omega)}\leq 
\frac{C}{h}\|R_h \nabla u -R_h\nabla v_h\|_{L^2(\Omega)}.
\end{eqnarray*}
We use the projection property of $R_h$ to write 
\begin{eqnarray*}
  \|R_h \bsigma -R_h\nabla v_h\|_{H^1(\Omega)}\leq  \frac{C}{h}
\|R_h(\nabla u -R_h\nabla v_h)\|_{L^2(\Omega)}.
\end{eqnarray*}
Now using the fact that $R_h$ is stable in $L^2$-norm, we have 
\[ \|R_h \bsigma -R_h\nabla v_h\|_{H^1(\Omega)}\leq  \frac{C}{h}
\|\nabla u -R_h\nabla v_h\|_{L^2(\Omega)}.\]
Since Theorem \ref{thnew1}  yields 
\begin{equation}
\|\nabla u -R_h\nabla v_h\|_{L^2(\Omega)} \leq C h^{1+\rho}\|u\|_{H^3(\Omega)},
\end{equation}
we have 
\[ \|\bsigma -\btau_h\|_{H^1(\Omega)} \leq  C h^{\rho} \|u\|_{H^3(\Omega)}.
\]
\end{proof}
We combine the result of Theorems \ref{th1} and \ref{th2} to 
get the final result. 
\begin{corollary}
Let $u$ be the solution of continuous problem \eqref{lspline} 
with $u \in H^4(\Omega)$, $\bsigma =\nabla u$ and $\bphi=\alpha\Delta\bsigma$,
 and $u_h$ be that of discrete 
problem \eqref{mvarp} with $\bsigma_h = R_h\nabla u_h$ and 
$\cV_h = \cQ_h \times [\cL_h]^d$
Then there exists a constant $C>0$ independent of 
the mesh-size $h$ so that 
\begin{eqnarray*}
\|(u-u_h,\bsigma-\bsigma_h)\|_A \leq 
Ch^{\rho}\left(\|u\|_{H^3(\Omega)}+|\bphi|_{H^1(\Omega)}\right).
\end{eqnarray*}
\end{corollary}
Thus $u_h$ and $\bsigma_h$ converge to 
$u$ and $\bsigma$ with a convergence rate of $O(h^{\rho})$. As 
$\rho \leq 1$, this rate may not be optimal.
\section{Conclusion}
We have presented a stabilized mixed finite element method for 
approximating thin plate splines in two and three dimensions.
The mixed formulation introduces two additional vector variables -- 
gradient of the smoother and Lagrange multiplier -- as unknowns.  
In order to be able to eliminate these variables in an 
efficient way, we propose to use a pair of finite element bases 
satisfying a biorthogonality property for discretizing 
the gradient and the Lagrange multiplier. We have shown 
convergence of the finite element approximation 
to the solution of thin plate splines.
\bibliographystyle{plain}
\bibliography{/home/bishnu/papers/total}
\end{document}